\documentclass{amsart}
\usepackage{setspace}
\usepackage{a4}
\usepackage{amsthm}
\usepackage{latexsym}
\usepackage{amsfonts}
\usepackage{graphicx}
\usepackage{textcomp}
\usepackage{cite}
\usepackage{enumerate}
\usepackage{amssymb}
\usepackage{hyperref}
\usepackage{amsmath}
\usepackage{tikz}
\usepackage[mathscr]{euscript}
\usepackage{mathtools}
\newtheorem{theorem}{Theorem}[section]

\newtheorem{definition}[theorem]{Definition}

\newtheorem{problem}[theorem]{Problem}

\title{This is the title}
\usepackage{fancyhdr}

\begin{document}
\hrule\hrule\hrule\hrule\hrule
\vspace{0.3cm}	
\begin{center}
{\bf{Noncommutative Equiangular Lines: van Lint-Seidel Relative and Gerzon Universal Bounds}}\\
\vspace{0.3cm}
\hrule\hrule\hrule\hrule\hrule
\vspace{0.3cm}
\textbf{K. Mahesh Krishna}\\
School of Mathematics and Natural Sciences\\
Chanakya University Global Campus\\
NH-648, Haraluru Village\\
Devanahalli Taluk, 	Bengaluru  Rural District\\
Karnataka State, 562 110, India \\
Email: kmaheshak@gmail.com\\

Date: \today
\end{center}

\hrule\hrule
\vspace{0.5cm}
\textbf{Abstract}: We introduce the notion of noncommutative equiangular lines and derive noncommutative versions of  fundamental van Lint-Seidel relative and Gerzon universal bounds.



\hrule

\hrule
\section{Introduction}
Let $d \in \mathbb{N}$ and $\gamma \in  [0,1]$. Recall that  a collection $\{\tau_j\}_{j=1}^n$  of unit vectors  in $\mathbb{R}^d$ or $\mathbb{C}^d$	is said to be \textbf{$\gamma$-equiangular  lines} \cite{LEMMENSSEIDEL, HAANTJES}  if 
\begin{align*}
	|\langle \tau_j, \tau_k\rangle|=\gamma, \quad \forall 1\leq j, k \leq n, j \neq k.
\end{align*}	
Two fundamental problems associated with equiangular lines are the following.
\begin{problem}\label{1} 
	Given $d \in \mathbb{N}$ and $\gamma \in  [0,1]$, what is the upper bound on $n$ such that there exists a collection $\{\tau_j\}_{j=1}^n$  of  $\gamma$-equiangular lines in $\mathbb{R}^d$?
\end{problem}
\begin{problem}\label{2} 
	For which parameters $d, n \in \mathbb{N}$ and $\gamma \in  [0,1]$, there exists a collection $\{\tau_j\}_{j=1}^n$  of  $\gamma$-equiangular lines in $\mathbb{R}^d$? If exists, how to construct them?
\end{problem}
Two  answers to Problem (\ref{1}) which are   fundamental driving force in the study of equiangular lines are the following results of van Lint and Seidel \cite{VANLINTSEIDEL, LEMMENSSEIDEL} and Gerzon \cite{WALDONBOOK}.
\begin{theorem} \cite{VANLINTSEIDEL, LEMMENSSEIDEL, GLAZYRINYU, GODSILROYLE} \label{VS} (\textbf{van Lint-Seidel Relative Bound}) 
	Let $\{\tau_j\}_{j=1}^n$  	be  $\gamma$-equiangular  lines in $\mathbb{R}^d$ or $\mathbb{C}^d$. 
	Then 
	\begin{align*}
		n(1-d\gamma^2) \leq d(1-\gamma^2).
	\end{align*}
	In particular, if 
	\begin{align*}
		\gamma < \frac{1}{\sqrt{d}},
	\end{align*}
	then 
	\begin{align}\label{LS}
		n\leq \frac{d(1-\gamma^2)}{1-d\gamma^2}.
	\end{align}
\end{theorem}
\begin{theorem} \cite{WALDONBOOK} \label{GBT}
	(\textbf{Gerzon Universal Bound}) 
	\begin{enumerate}[\upshape(i)]
		\item 
		Let $\{\tau_j\}_{j=1}^n$  	be  $\gamma$-equiangular  lines in $\mathbb{R}^d$. 
		Then 
		\begin{align}\label{GB}
			n \leq \frac{d(d+1)}{2}.
		\end{align}	
		\item 
		Let $\{\tau_j\}_{j=1}^n$  	be  $\gamma$-equiangular  lines in $\mathbb{C}^d$. 
		Then 
		\begin{align}\label{CGB}
			n \leq d^2.
		\end{align}	
	\end{enumerate}
\end{theorem}
Bounds in Inequalities  (\ref{LS}),  (\ref{GB}) and (\ref{CGB}) are  improved for various special values of $\gamma$, $n$ and $d$ (including asymptotic). Regarding Problem (\ref{2}), various special cases are solved but full generality remains open. We refer \cite{BUKH, JIANGTIDORYAOZHANGZAO, BALLADRAXLERKEEVASHSUDAKOV, BARGYU, DECAEN, MALOZEMOPEVNYI, NEUMAIER, JIANGPOLYANSKII, DELAATMACHADODEOLIVEIRAFERNANDAVELLENTIN, GREAVESKOOLENMUNEMASASZOLLOSO, GREAVSSYATRIADIYATSYNA, OKUDAYU, YU, KINGTANG, CONWAYHARDINSLOANE, GREAVESSYATRIADI2, GREAVESSYATRIADIYATSYNA2, SEIDELHANDBOOK, GREAVES7, KOORNWINDER, LINYU, LINYU2, ZHAO, COUTINHOGODSILSHIRAZIZHAN}
 for a look on these achievements in the real case and \cite{JEDWABWIEBE, JEDWABWIEBE2, GODSILROY, SUSTIKTROPPDHILLON, STROHMER, HOFFMANSOLAZZO, BROUWERHAEMERS}
    in the complex case. In the real case, it is known that equality is not always attained for  every $d$ in Inequality (\ref{GB}).   Celebrated Zauner conjecture says that  there exists a   $d^2$ $\frac{1}{\sqrt{d+1}}$-equiangular lines in $\mathbb{C}^d$ for every $d$ satisfying equality in Inequality (\ref{CGB}) \cite{ZAUNER, RENESBLUMEKOHOUTSCOTTCAVES, APPLEBY, APPLEBYBENGTSSONFLAMMIAGOYENECHE, KOPP, APPLEBYFLAMMIAMCCONNELLYARD, SCOTT}.

We wish to record that there is  a vector space version of equiangular lines \cite{GREAVESIVERSONJASPERMIXON2, GREAVESIVERSONJASPERMIXON}  and  there is also a notion of  equiangular subspaces of Euclidean spaces \cite{BALLASUDAKOV, BALLADRAXLERKEEVASHSUDAKOV2}. Recently, p-adic version of equiangular lines have been introduced \cite{MAHESH}.

As noncommutative geometry plays major role in modern Mathematics, our fundamental objective is to initiate a study of noncommutative equiangular lines. As most important results in the theory of equiangular lines are van Lint-Seidel relative and Gerzon universal bounds, we derive noncommutative versions of them. 

\section{Noncommutative Equiangular Lines}
We want the notion of Hilbert C*-modules. They are first introduced by Kaplansky \cite{KAPLANSKY} for modules over commutative C*-algebras and later developed for modules over arbitrary C*-algebras by Paschke  \cite{PASCHKE} and Rieffel \cite{RIEFFEL}. A friendly introduction is given in the book \cite{WEGGEOLSEN}. We want only standard Hilbert C*-module. Let $\mathcal{A}$ be a  unital C*-algebra with identity $1$.
 For $d\in \mathbb{N}$, let $\mathcal{A}^d$ be the standard left Hilbert C*-module over $\mathcal{A}$ with inner product 
\begin{align*}
	\langle (a_j)_{j=1}^d, (b_j)_{j=1}^d\rangle \coloneqq \sum_{j=1}^{d}a_jb_j^*, \quad \forall (a_j)_{j=1}^d, (b_j)_{j=1}^d \in \mathcal{A}^d.
\end{align*}
Hence norm on $\mathcal{A}^d$ is 
\begin{align*}
	\|(a_j)_{j=1}^d\|_2\coloneqq \left\|\sum_{j=1}^{d}a_ja_j^*\right\|^\frac{1}{2}, \quad \forall (a_j)_{j=1}^d \in \mathcal{A}^d.
\end{align*}
 We  introduce two notions of equiangular lines for modules.
\begin{definition}\label{A}
Let $a \in \mathcal{A}$ be a positive	element.   A collection $\{\tau_j\}_{j=1}^n$ in  $\mathcal{A}^d$  is said to be \textbf{modular $a$-equiangular  lines} if following conditions hold.
\begin{enumerate}[\upshape(i)]
	\item $	\langle \tau_j, \tau_j \rangle =1,  \forall 1\leq j \leq n$.
	\item $	\langle\tau_j, \tau_k \rangle\langle\tau_k, \tau_j \rangle=a,  \forall 	1\leq j , k \leq n , j\neq k.$
\end{enumerate}
\end{definition}
\begin{definition}
	Let $\gamma\in [0,1].$   A collection $\{\tau_j\}_{j=1}^n$ in  $\mathcal{A}^d$  is said to be \textbf{modular $\gamma$-norm-equiangular  lines} if following conditions hold.
	\begin{enumerate}[\upshape(i)]
		\item $	\langle \tau_j, \tau_j \rangle =1,  \forall 1\leq j \leq n$.
		\item $	\|\langle\tau_j, \tau_k \rangle\|=\gamma,  \forall 	1\leq j , k \leq n , j\neq k.$
	\end{enumerate}
\end{definition}
It is clear that if $\{\tau_j\}_{j=1}^n$ is modular  $a$-equiangular  lines, then $\{\tau_j\}_{j=1}^n$ is modular  $\sqrt{\|a\|}$-norm-equiangular  lines. We now derive two modular versions  of Theorem \ref{VS} for modules over commutative C*-algebras.
\begin{theorem} (\textbf{Modular van Lint-Seidel Relative Bound}) \label{MVS}
Let $\mathcal{A}$ be a commutative unital C*-algebra. 	Let $\{\tau_j\}_{j=1}^n$  	be modular  $a$-equiangular  lines in $\mathcal{A}^d$.
	Then 
	\begin{align*}
		n(1-da) \leq d(1-a).
	\end{align*}
	In particular, if $1-da$ is invertible, then 
	\begin{align*}
		n\leq \frac{d(1-a)}{1-da}\leq d \left\|\frac{1-a}{1-da}\right\|.
	\end{align*}
	\end{theorem}
\begin{proof}
Let $\tau_j\coloneqq(a_1^{(j)}, a_2^{(j)}, \dots, a_d^{(j)}) $ for all $1\leq j\leq n$. Define 
	\begin{align*}
		B\coloneqq \sum_{j=1}^n\sum_{k=1}^n\langle \tau_j, \tau_k\rangle \langle \tau_k, \tau_j\rangle =n+(n^2-n)a.
	\end{align*}
Then using Cauchy-Schwarz inequality in Hilbert C*-modules, 
\begin{align*}
	B&= \sum_{j=1}^n\sum_{k=1}^n\left(\sum_{r=1}^{d}a_r^{(j)}(a_r^{(k)})^*\right)	\left(\sum_{s=1}^{d}a_s^{(k)}(a_s^{(j)})^*\right)\\
	&=\sum_{r, s=1}^{d}\left(\sum_{j=1}^na_r^{(j)}(a_s^{(j)})^*\right)\left(\sum_{k=1}^na_r^{(k)}(a_s^{(k)})^*\right)^*\\
	&=\sum_{r=1}^{d}\left(\sum_{j=1}^na_r^{(j)}(a_r^{(j)})^*\right)\left(\sum_{k=1}^na_r^{(k)}(a_r^{(k)})^*\right)^*+\sum_{r, s=1, r\neq s}^{d}\left(\sum_{j=1}^na_r^{(j)}(a_s^{(j)})^*\right)\left(\sum_{k=1}^na_r^{(k)}(a_s^{(k)})^*\right)^*\\
	&\geq \sum_{r=1}^{d}\left(\sum_{j=1}^na_r^{(j)}(a_r^{(j)})^*\right)\left(\sum_{k=1}^na_r^{(k)}(a_r^{(k)})^*\right)^*\\
	&\geq \frac{1}{d}\left(\sum_{r=1}^{d}\left(\sum_{j=1}^na_r^{(j)}(a_r^{(j)})^*\right)\cdot 1\right)\left(\sum_{s=1}^{d}1\cdot\left(\sum_{k=1}^na_s^{(k)}(a_s^{(k)})^*\right)^*\right)\\
	&=\frac{1}{d}\left(\sum_{j=1}^n\sum_{r=1}^{d}a_r^{(j)}(a_r^{(j)})^* \right)\left(\sum_{k=1}^n\sum_{s=1}^{d}a_s^{(k)}(a_s^{(k)})^*\right)\\
	&=\frac{1}{d}\left(\sum_{j=1}^{n}1\right)\left(\sum_{k=1}^{n}1\right)=\frac{n^2}{d}.
\end{align*}
Hence 
\begin{align*}
	(n-1) a+1\geq   \frac{n}{d} \implies 	d(n-1) a+d\geq n \implies d(1-a)\geq n(1-da).
\end{align*}
\end{proof}
\begin{theorem} 
Let $\mathcal{A}$ be a commutative unital C*-algebra.	Let $\{\tau_j\}_{j=1}^n$  	be modular  $\gamma$-norm-equiangular  lines in $\mathcal{A}^d$. Then 
		\begin{align*}
		n(1-d\gamma^2) \leq d(1-\gamma^2).
	\end{align*}
	In particular, if 
		\begin{align*}
		\gamma < \frac{1}{\sqrt{d}},
	\end{align*}
	then 
	\begin{align*}
		n\leq \frac{d(1-\gamma^2)}{1-d\gamma^2}.
	\end{align*}
\end{theorem}
\begin{proof}
	By using the proof of Theorem \ref{MVS}, 
	\begin{align*}
		 \frac{n^2}{d}&\leq \sum_{j=1}^n\sum_{k=1}^n\langle \tau_j, \tau_k\rangle \langle \tau_k, \tau_j\rangle = \sum_{j,k=1, j\neq k}^n\langle \tau_j, \tau_k\rangle \langle \tau_k, \tau_j\rangle +\sum_{j=1}^n\langle \tau_j, \tau_j\rangle^2\\
		 &=\sum_{j,k=1, j\neq k}^n\langle \tau_j, \tau_k\rangle \langle \tau_k, \tau_j\rangle +n\leq \sum_{j,k=1, j\neq k}^n\|\langle \tau_j, \tau_k\rangle \|^2 +n=(n^2-n)\gamma^2+n.
	\end{align*}
\end{proof}

Proof of previous theorems used commutativity of C*-algebra. We are unable to derive noncommutative versions of them. However,  we introduce a subclass of norm-equiangular lines for which we derive modular relative bound.
\begin{definition}
	Let $\gamma\in [0,1].$   A collection $\{\tau_j\}_{j=1}^n$ in  $\mathcal{A}^d$  is said to be \textbf{special modular $\gamma$-norm-equiangular  lines} if following conditions hold.
	\begin{enumerate}[\upshape(i)]
		\item $	\langle \tau_j, \tau_j \rangle =1,  \forall 1\leq j \leq n$.
		\item $	\|\langle\tau_j, \tau_k \rangle\|=\gamma,  \forall 	1\leq j , k \leq n , j\neq k.$
		\item 
		\begin{align*}
			n^2 \leq d \sum_{j=1}^n\sum_{k=1}^n\langle \tau_j, \tau_k\rangle \langle \tau_k, \tau_j\rangle.
		\end{align*}
	\end{enumerate}
\end{definition}
\begin{theorem} (\textbf{Noncommutative van Lint-Seidel Relative Bound})
Let $\mathcal{A}$ be a  unital C*-algebra.	Let $\{\tau_j\}_{j=1}^n$  	be special modular  $\gamma$-norm-equiangular  lines in $\mathcal{A}^d$. Then 
\begin{align*}
	n(1-d\gamma^2) \leq d(1-\gamma^2).
\end{align*}
In particular, if $\gamma<\sqrt{d^{-1}}$, then 
\begin{align*}
	n\leq \frac{d(1-\gamma^2)}{1-d\gamma^2}.
\end{align*}
\end{theorem}
\begin{proof}
\begin{align*}
	n^2 \leq d \sum_{j=1}^n\sum_{k=1}^n\langle \tau_j, \tau_k\rangle \langle \tau_k, \tau_j\rangle\leq (n^2-n)\gamma^2+n.
\end{align*}	
\end{proof}
Next we derive modular Gerzon bound for modules over C*-algebras with invariant basis number  (IBN) property. These C*-algebras are studied in \cite{GIPSON}.
\begin{theorem} (\textbf{Modular Gerzon Universal Bound}) 
Let $\mathcal{A}$ be a  unital C*-algebra with IBN. Let $\{\tau_j\}_{j=1}^n$  	be  modular $a$-equiangular  lines in $\mathcal{A}^d$. If $1-a$ is invertible, then
		\begin{align}\label{MG}
			n \leq d^2.
		\end{align}	
\end{theorem}
\begin{proof}
For $1\leq j \leq n$, define 
\begin{align*}
	\tau_j \otimes \tau_j: \mathcal{A}^d \ni x \mapsto(\tau_j \otimes \tau_j)x\coloneqq\langle x, \tau_j\rangle \tau_j \in \mathcal{A}^d.
\end{align*}
We show that the collection $\{\tau_j \otimes \tau_j\}_{j=1}^n$   is linearly independent  over  $\mathcal{A}$. Let $c_1, \dots,  c_n \in \mathcal{A}$ be such that 
\begin{align}\label{12}
	\sum_{j=1}^{n}c_j(\tau_j \otimes \tau_j)=0.
\end{align}
Let $1\leq k \leq n$ be fixed. Then previous equation gives 
\begin{align*}
		0=\sum_{j=1}^{n}c_j(\tau_j \otimes \tau_j)\tau_k=\sum_{j=1}^{n}c_j\langle \tau_k, \tau_j\rangle \tau_j.
\end{align*}
Taking inner product with $\tau_k$ gives 
\begin{align*}
	0&=\sum_{j=1}^{n}c_j\langle \tau_k, \tau_j\rangle \langle \tau_j, \tau_k \rangle=\sum_{j=1, j\neq k}^nc_j\langle \tau_k, \tau_j\rangle \langle \tau_j, \tau_k\rangle +c_k\langle \tau_k, \tau_k\rangle^2\\
	&=\sum_{j=1, j\neq k}^nc_ja+c_k=\sum_{j=1}^nc_ja-c_ka+c_k=\sum_{j=1}^nc_ja+c_k(1-a).
\end{align*}
Now define 
\begin{align*}
	c\coloneqq \frac{-1}{1-a}\sum_{j=1}^{n}c_ja.
\end{align*}
Then $c_k=c$ for all $1\leq k \leq n$. We wish to show that $c=0$. Now Equation (\ref{12}) gives 
\begin{align*}
	0=	\sum_{j=1}^{n}c(\tau_j \otimes \tau_j).
\end{align*}
Let $1\leq k \leq n$ be fixed. Then previous equation gives 
\begin{align*}
	0=\sum_{j=1}^{n}c(\tau_j \otimes \tau_j)\tau_k=\sum_{j=1}^{n}c\langle \tau_k, \tau_j\rangle \tau_j.
\end{align*}
Taking inner product with $\tau_k$ gives 
\begin{align*}
	0&=\sum_{j=1}^{n}c\langle \tau_k, \tau_j\rangle \langle \tau_j, \tau_k \rangle=\sum_{j=1, j\neq k}^nc\langle \tau_k, \tau_j\rangle \langle \tau_j, \tau_k\rangle +c\langle \tau_k, \tau_k\rangle^2\\
	&=\sum_{j=1, j\neq k}^nca+c=(n-1)ca+c=c((n-1)a+1).
\end{align*}
Since $a\geq0$, $((n-1)a+1)$ is invertible, hence $c=0$. Note that  the matrix of $\tau_j \otimes \tau_j$ is $\tau_j \tau_j^*$ (viewing $\tau_j$ as column vector). The rank of $d$ by $d$ matrices over $\mathcal{A}$ is $d^2$. Since $\{\tau_j \otimes \tau_j\}_{j=1}^n$ is linearly independent and $\mathcal{A}$ has IBN property, $n\leq d^2$. 
\end{proof}
We can slightly generalize Definition \ref{A}. In the vector space case, such a generalization was done by Greaves,  Iverson, Jasper and Mixon \cite{GREAVESIVERSONJASPERMIXON}.
\begin{definition}\label{AB}
Let $a, b \in \mathcal{A}$ be  positive elements.   A collection $\{\tau_j\}_{j=1}^n$ in  $\mathcal{A}^d$  is said to be \textbf{modular $(a,b)$-equiangular  lines} if following conditions hold.
\begin{enumerate}[\upshape(i)]
	\item $	\langle \tau_j, \tau_j \rangle =b,  \forall 1\leq j \leq n$.
	\item $	\langle\tau_j, \tau_k \rangle\langle\tau_k, \tau_j \rangle=a,  \forall 	1\leq j , k \leq n , j\neq k.$
\end{enumerate}	
\end{definition}
Note that $b$ is necessarily positive but need not be invertible. Thus we may not able to reduce Definition \ref{AB} to Definition \ref{A}. By modifying earlier proofs, we easily get following theorems. 
\begin{theorem}
Let $\mathcal{A}$ be a commutative unital C*-algebra. 	Let $\{\tau_j\}_{j=1}^n$  	be modular  $(a,b)$-equiangular  lines in $\mathcal{A}^d$.
Then 
\begin{align*}
	n(b^2-da) \leq d(b-a).
\end{align*}
In particular, if $b^2-da$ is invertible, then 
\begin{align*}
	n\leq \frac{d(b-a)}{b^2-da}\leq d \left\|\frac{b-a}{b^2-da}\right\|.
\end{align*}	
\end{theorem}
\begin{theorem}
	Let $\mathcal{A}$ be a  unital C*-algebra with IBN. Let $\{\tau_j\}_{j=1}^n$  	be  modular $(a,b)$-equiangular  lines in $\mathcal{A}^d$. If $b^2-a$ and $(n-1)a+b^2$ are invertible, then
	\begin{align*}
		n \leq d^2.
	\end{align*}	
\end{theorem}
Similar to scalar  equiangular lines and Zauner conjecture, we formulate following problem.
\begin{problem}\label{NZ}
Let $\mathcal{A}$ be a  unital C*-algebra.	For which parameters $d, n \in \mathbb{N}$ and $a \in \mathcal{A}$, there exists a collection $\{\tau_j\}_{j=1}^n$  of  modular $a$-equiangular lines in $\mathcal{A}^d$? In particular, for which $d$/for all $d$, there is $d^2$ modular $\frac{1}{d+1}$-equiangular lines in $\mathcal{A}^d$ (making equality in Inequality (\ref{MG}))?
\end{problem}
Note that Problem \ref{NZ} contains Zauner conjecture (whenever $\mathcal{A}=\mathbb{C}$).

\section{Acknowledgments}
This paper has been partially developed when the author attended the workshop “Quantum groups, tensor categories and quantum field theory”, held in the University of Oslo, Norway from January 13 to 17, 2025. This event was organized by the University of Oslo, Norway and funded by the Norwegian Research Council through the “Quantum Symmetry” project.

 \bibliographystyle{plain}
 \bibliography{reference.bib}

\end{document}